\documentclass[a4paper,12pt]{article}
\usepackage{amsmath,amsthm,amssymb,amsfonts}
\usepackage{bbm,bm}
\usepackage{latexsym}
\usepackage{mathrsfs}
\usepackage{threeparttable}
\usepackage{tabularx}
\usepackage{booktabs}
\usepackage{graphicx}
\usepackage{epstopdf}
\usepackage{color}
\usepackage{indentfirst}
\usepackage{mathrsfs}
\usepackage{geometry}
\usepackage{caption}
\usepackage{lineno}
\usepackage{cleveref}
\usepackage{abstract}
\usepackage{enumerate,mdwlist}
\usepackage{cases}
\usepackage{subfigure}

\usepackage[numbers,sort&compress]{natbib}

\theoremstyle{plain}
\newtheorem{The}{Theorem}[section]    
\newtheorem{Lem}[The]{Lemma}

\newtheorem{Rem}[The]{Remark}

\theoremstyle{definition}
\newtheorem{Def}{Definition}

\numberwithin{equation}{section} 

\makeatletter

\newcommand{\Rmnum}[1]{\expandafter\@slowromancap\romannumeral #1@}
\makeatother
\title{Star Structure Connectivity of Folded hypercubes and Augmented cubes}

\author{Lina Ba, Hailun Wu and Heping Zhang\thanks{Corresponding author.} }
\date{{\small School of Mathematics and Statistics, Lanzhou University,
 Lanzhou, Gansu 730000, P.R. China}\\
{\small E-mails:\ baln19@lzu.edu.cn, zhanghp@lzu.edu.cn}}

\begin{document}

\maketitle
\begin{abstract}
 The connectivity is an important parameter to evaluate the robustness of a network. As a generalization, structure connectivity and substructure connectivity of graphs were proposed. For connected graphs $G$ and $H$, the $H$-structure connectivity $\kappa(G; H)$ (resp. $H$-substructure connectivity $\kappa^{s}(G; H)$) of $G$ is the minimum cardinality of a set of subgraphs $F$  of $G$ that each  is isomorphic to $H$ (resp.  to a connected subgraph of $H$) so that  $G-F$ is disconnected or the singleton. As popular variants of hypercubes, the $n$-dimensional folded hypercubes $FQ_{n}$ and augmented cubes $AQ_{n}$ are attractive interconnected network prototypes for multiple processor systems. In this paper, we obtain that
$\kappa(FQ_{n};K_{1,m})=\kappa^{s}(FQ_{n};K_{1,m})=\lceil\frac{n+1}{2}\rceil$ for $2\leqslant m\leqslant n-1$, $n\geqslant 7$, and $\kappa(AQ_{n};K_{1,m})=\kappa^{s}(AQ_{n};K_{1,m})=\lceil\frac{n-1}{2}\rceil$ for $4\leqslant m\leqslant \frac{3n-15}{4}$.
  \vskip 0.1in

    \noindent {\bf Keywords:} \ Interconnection network;  Structure connectivity; Star; Folded hypercube;  Augmented cube.

    \vskip 0.1 in


    \medskip
\end{abstract}
\section{Introduction}
As we know that an interconnection network is usually represented by an undirected simple graph $G=(V(G), E(G))$, where $V(G)$ and $E(G)$ are the vertex set and the edge set, respectively, of $G$. Each vertex  represents a processor or node and every edge a communication link. Throughout this paper, graphs are finite, simple, and undirected. Two vertices $u$ and $v$ of $G$ are adjacent if they are the end-vertices of an edge. A neighbor of a vertex $x$ of $G$ means a vertex of $G$ adjacent to $x$. The neighborhood of a vertex $x$ in $G$ is the set of neighbors of $x$, denoted by $N_{G}(x)=\{y|xy\in E(G), y\in V(G)\}$. The neighborhood of a vertex subset $A$ in $G$ is denoted by $N_{G}(A)=\cup_{x\in A}N_{G}(x)-A$.
A graph $H$ is a \emph{subgraph} of $G$ if $V (H)\subseteq V(G)$ and $E(H) \subseteq E(G)$. A \emph{spanning subgraph} of $G$ is a subgraph with vertex set $V(G)$.

One of the most basic measure of an interconnection network's fault-tolerant capability is the connectivity. The \emph{connectivity} $\kappa(G)$ of $G$ is defined as the minimum cardinality of a vertex set $S$ such that $G-S$ is disconnected or has only one vertex. In order to provide more accurate measures for fault tolerance of large-scale parallel processing systems, Harary \cite{24} considered \emph{$\rho$-conditional connectivity} $\kappa(G;\rho)$. It is defined as the minimum cardinality of a vertex set whose deletion can disconnect $G$ but every remaining component would still possess the property $\rho$. Based on the $\rho$-conditional connectivity, Fabrega et al. \cite {3} proposed $g$-extra connectivity.  For a connected non-complete graph $G$ and a non-negative integer $g$, a vertex cut $S$ of $G$ is an $g$-extra cut if $G-S$ is disconnected and every component of $G-S$ has more than $g$ vertices. The \emph{$g$-extra connectivity } $\kappa_{g}(G)$ of $G$ is defined as the minimum cardinality of $g$-extra cut of $G$. Further, a number of extensions were introduced including  restricted connectivity, component connectivity and strongly Menger connectivity. However, in consideration of network reliability and fault-tolerance, the neighbors of a faulty node might be more vulnerable. Lin et al. \cite {4} introduced the structure connectivity and substructure connectivity.
A set $F$ of subgraphs of $G$ is a subgraph cut of $G$ if $G-V (F)$ is disconnected or a trivial graph, where $V (F)$ denotes the set of vertices contained in any element of $F$. Let $H$ be any connected subgraph of $G$. Then $F$ is an $H$-structure cut if $F$ is a subgraph cut of $G$ and every subgraph in $F$ is isomorphic to $H$. The \emph{$H$-structure connectivity} of $G$, denoted by $\kappa(G; H)$, is the minimum cardinality of all $H$-structure cuts of $G$.
Furthermore, $F$ is an $H$-substructure cut if $F$ is a subgraph cut of $G$ such that each element of $F$ is isomorphic to a connected subgraph of $H$. The \emph{$H$-substructure connectivity} of $G$, denoted by $\kappa^s(G; H)$, is the minimum cardinality of all $H$-substructure cuts of $G$. By definition, $\kappa^{s}(G,H)\leqslant \kappa(G,H)$. Note that $K_{1}$-structure connectivity  reduces to the classical vertex connectivity.

In recent years,  much of the work has been focused on structure connectivity of hypercube and its variants (see, for example, \cite{6, 7, 10, 12, 9, 26}). The $n$-dimensional folded hypercube $FQ_n$ is a highly symmetric graph as a underlying topology of several parallel systems, such as ATM Switches \cite {8}, PM21 networks \cite {5} and 3D-FolH-NOC network \cite {2}. In \cite{19}, we showed that $\kappa(FQ_{n};K_{1,m})=\kappa^{s}(FQ_{n};K_{1,m})= \lceil\frac{n+1}{2}\rceil$ and $\kappa(Q_{n};K_{1,m})=\kappa^{s}(Q_{n};K_{1,m})= \lceil\frac{n}{2}\rceil$ for all integers $n$ larger than  $m$ in square scale, where $Q_n$ is $n$-dimensional hypercube and $K_{1,m}$ is a star with $m$ leaves $(m\geqslant 1)$. These results just partly solved the $K_{1,m}$-structure(substructure) connectivity of $Q_n$ and $FQ_n$. And Lin et al. \cite {25} completely solved the $K_{1,m}$-structure(substructure) connectivity for $n$-dimensional  crossed cube  $CQ_{n}$ and hypercube $Q_{n}$. The results is that  $\kappa(Q_{n},K_{1,1})=\kappa^{s}(Q_{n},K_{1,1})=n-1$, $\kappa(Q_{n},K_{1,m})=\kappa^{s}(Q_{n},K_{1,m})=\lceil\frac{n}{2}\rceil$ for $2\leqslant m\leqslant n$, $n\geqslant 4$ and  $\kappa(CQ_{n},K_{1,1})=\kappa^{s}(CQ_{n},K_{1,1})=n-1$, $\kappa(CQ_{n},K_{1,m})=\kappa^{s}(CQ_{n},K_{1,m})=\lceil\frac{n}{2}\rceil$ for $2\leqslant m\leqslant n$, $n\geqslant 4$. Further, they  established sharp lower bounds on $K_{1,m}$-structure(substructure) connectivity of the family of hypercube-like networks for $1\leqslant m\leqslant n$.  Since there are numerous variants of hypercube, we can find that there are still some variants of hypercube do not belong to the family of hypercube-like networks, such as folded hypercube and augmented cube.
For $n$-dimensional augmented cube $AQ_n$, Kan et al. \cite{23} showed  $\kappa^s(AQ_n; K_{1,m})=\kappa(AQ_n; K_{1,m})=\lceil\frac{2n-1}{1+m}\rceil$ for $1\leqslant m\leqslant 3$, $n\geqslant 4$ and $\kappa(AQ_n; K_{1,m})=\kappa^s(AQ_n; K_{1,m})=\lceil\frac{n-1}{2}\rceil$ for $4\leqslant m\leqslant 6$, $n\geqslant 6$.
 We know that general $K_{1,m}$-structure connectivity of $FQ_n$ and $AQ_n$ have not been completely solved.  In this paper, we show that $\kappa(FQ_{n};K_{1,m})=\kappa^{s}(FQ_{n};K_{1,m})=\lceil\frac{n+1}{2}\rceil$ for $2\leqslant m\leqslant n-1$, $n\geqslant 7$ and $\kappa(AQ_{n};K_{1,m})=\kappa^{s}(AQ_{n};K_{1,m})=\lceil\frac{n-1}{2}\rceil$ for $4\leqslant m\leqslant \frac{3n-15}{4}$.

The rest parts of this paper are organized as follows. Section 2 gives some definitions and some known results on hypercubes. Sections 3 and 4 present the star-structure connectivity of folded hypercubes and augmented cubes, respectively. Section 5 summarizes our results obtained in this paper and points out some unsolved problems.

\section{Preliminaries}

For integer $n\geqslant 1$, the $n$-dimensional hypercube $Q_n$ is a graph: each vertex $x$ corresponds to an $n$-bit binary string $x_n x_{n-1} \ldots x_{1}$.  For $i\in\{1,2, \ldots, n \}$, $x^i$ denotes the binary string $x_n \ldots \overline{x_i}\ldots x_1$ where $\overline{x_i} = 1-x_i$.  Two vertices $x$ and $y$ in $Q_n$ are adjacent if and only if $y=x^i$ for some $i\in\{1,2, \ldots, n \}$. In this case, we call $x^i$ is the neighbor of $x$ in dimension $i$. Similarly, we set $x^{i,j}$ is the neighbor of $x^i$ in dimension $j$ and $x^{i,j,k}$ is the neighbor of $x^{i,j}$ in dimension $k$.

\begin{Def}\cite{1}
The $n$-dimensional folded hypercube $FQ_n$ is a graph obtained from hypercube $Q_n$ by adding $2^{n-1}$ edges, each of them is  between vertices $x=x_n x_{n-1} \ldots x_1$ and $\overline{x}=\overline{x_n}$ $\overline{x_{n-1}}  \ldots  \overline{x_1}$ where $\overline{x_i} = 1-x_i$.
\end{Def}

For any vertex $x=x_n x_{n-1}\ldots x_1$ in $FQ_n$, we set $(\overline{x})^i$ be the neighbor of $\overline{x}$ in dimension $i$ and $(\overline{x})^{i,j}$ is the neighbor of $(\overline{x})^i$ in dimension $j$.

The $n$-dimensional augmented cube $AQ_n$$(n\geqslant1)$ can be defined recursively as follows.
\begin{Def}\cite{21}
Let $n\geqslant 1$ be an integer. The augmented cube $AQ_n$ of dimension $n$ has $2^n$ vertices, each labeled by an $n$-bit binary string $a_na_{n-1}\ldots a_1$. We define $AQ_1=K_2$. For $n\geqslant 2$, $AQ_n$ is obtained by taking two copies of the augmented cube $AQ_{n-1}$, denoted by $AQ_{n-1}^0$ and $AQ_{n-1}^1$, and adding $2\times 2^{n-1}$ edges between the two as follows:

Let $V(AQ_{n-1}^0)=\{0a_{n-1}a_{n-2}\ldots a_1 | a_i= 0~\rm{or}~1\}$ and $V(AQ_{n-1}^1)=\{1b_{n-1}b_{n-2}\ldots b_1 | b_i= 0~\rm{or}~1\}$. A vertex $u=0a_{n-1}a_{n-2}\ldots a_1$ of $AQ_{n-1}^0$ is adjacent to a vertex $v=1b_{n-1}b_{n-2}\ldots b_1$ of $AQ_{n-1}^1$ if and only if either $a_i=b_i$ for all $1\leqslant i\leqslant n-1$ or $a_i=\overline{b_i}$ for all $1\leqslant i\leqslant n-1$.
\end{Def}

We write this recursive construction of $AQ_n$ symbolically as $AQ_n=AQ_{n-1}^0 \otimes AQ_{n-1}^1$. For an $n$-bit binary string $u=u_nu_{n-1}\ldots u_1$, we use $\overline{u}^i$ to denote the binary string
 $u_n \ldots u_{i+1}\overline{u_i}\ldots \overline{u_1}$ which differs with $u$  from the first to the $i$th bit positions. It is clear that $u^1 = \overline{u}^1$. We call $uu^i$ a hypercube edge and $u\overline{u}^i$ a complement edge in $AQ_n$.

It has been shown that $AQ_n$$ (n\neq3)$ is $(2n-1)$-regular, $(2n-1)$-connected and vertex-symmetric \cite{21}. $FQ_n$ and $AQ_n$ are obtained from hypercube $Q_n$ by adding some specific edges. For example, $Q_3$, $AQ_3$ and $FQ_3$ are illustrated in Figure 1.
\begin{figure}[!htbp]\label{Q3}
\centering
\subfigure[$Q_3$ ]{\includegraphics[totalheight=6cm]{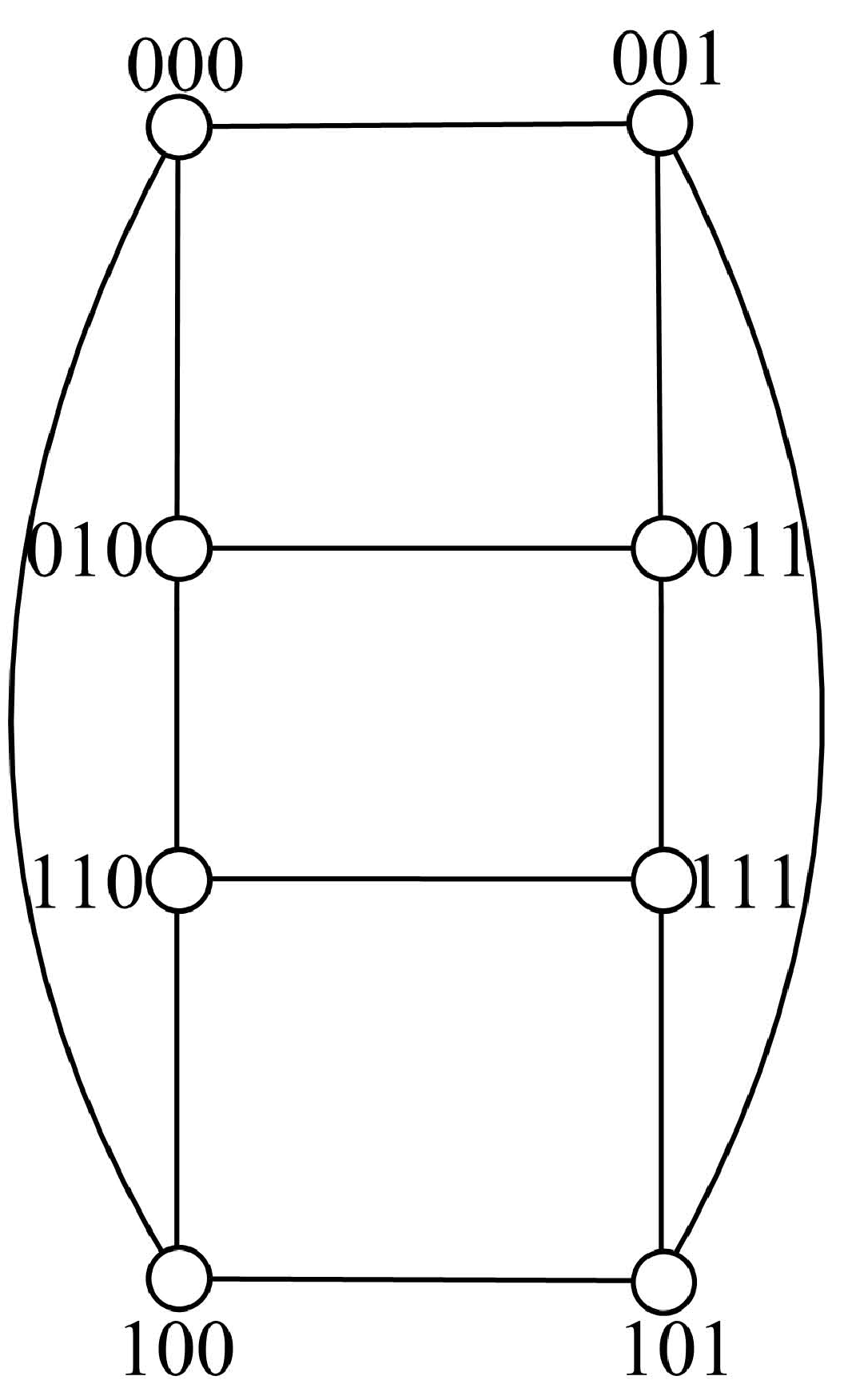}}~~~~~~~~
\subfigure[$FQ_3$ ]{\includegraphics[totalheight=6cm]{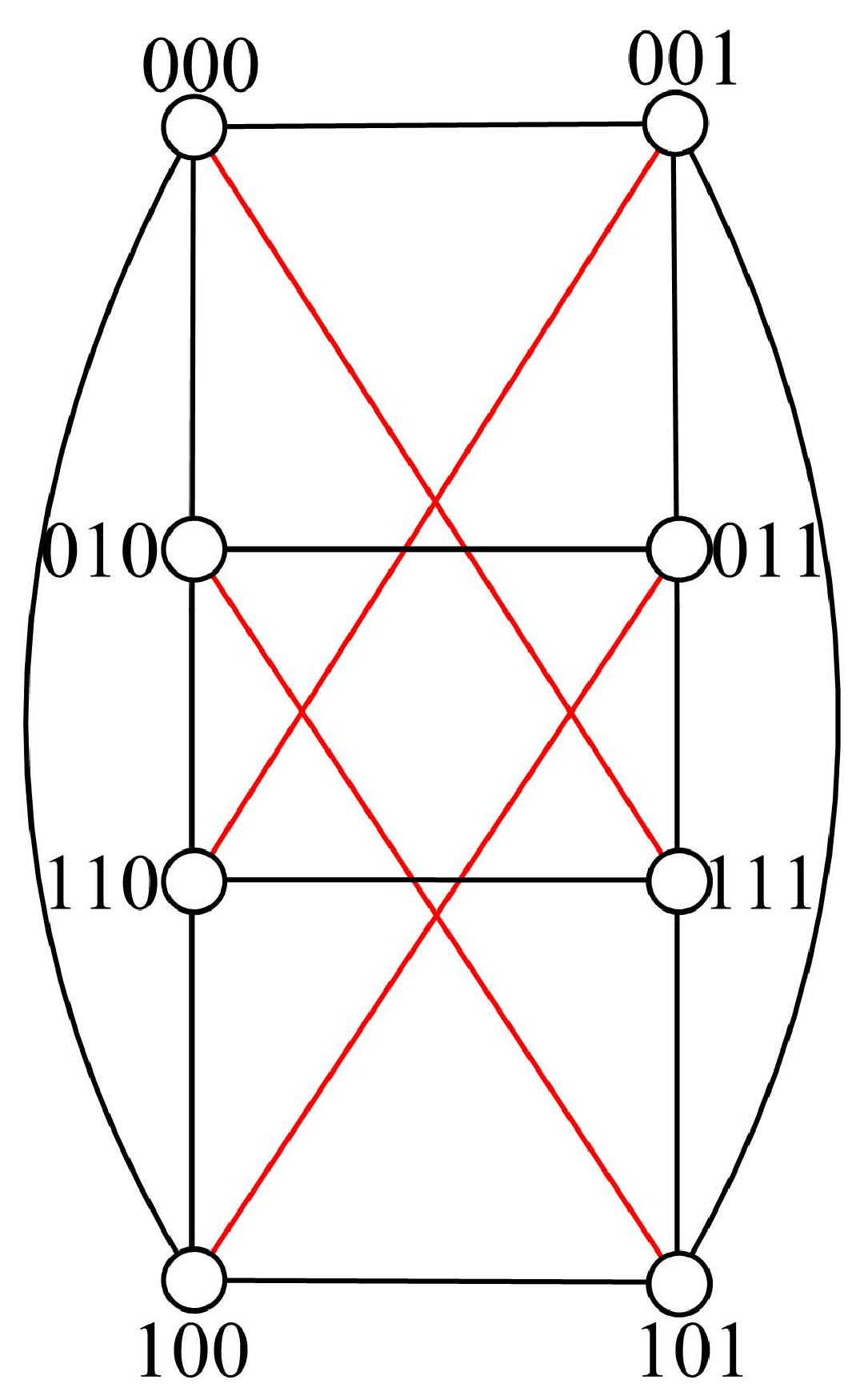}}~~~~~~~~
\subfigure[$AQ_3$]{\includegraphics[totalheight=6cm]{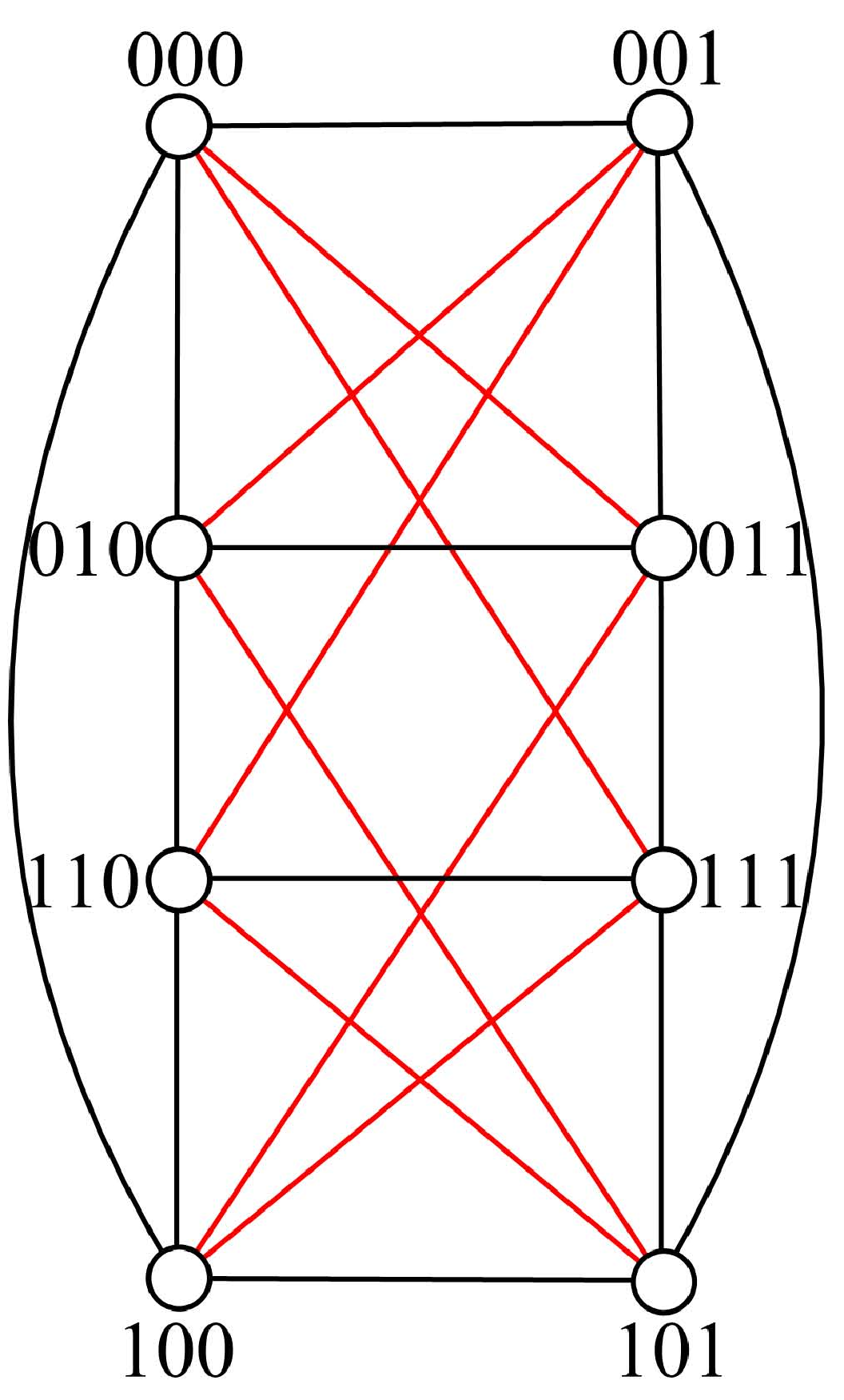}}
\caption{$Q_3$, $FQ_3$ and $AQ_3$.}
\end{figure}

For convenience, we define two functions $f(x)$ and $g(x)$ as follows:
\begin{equation*}\label{}
f(x)=-\frac{x^2}{2}+(n-\frac{1}{2})x+1;~~~~~~~
g(x)=-\frac{x^2}{2}+(2n-\frac{3}{2})x+2-n^2.
\end{equation*}

Now, we introduce the following important results of $n$-dimensional hypercubes $Q_n$.

\begin{Lem}\label{Qn}\emph{\cite {18, 17}}Let $S$ be a vertex subset of $Q_n$ with $n\geqslant 4$.
\begin{enumerate}[(i)]
\item \label{cond 1}
If $|S|<f(k)$ and $1\leqslant k\leqslant n-2$, then $Q_n-S$ contains exactly one large component of order at least $2^n-|S|-(k-1)$ and the small components having at most $k-1$ vertices in total,

\item \label{cond 2}
If $|S|<f(k)$ and $n-1\leqslant k\leqslant n+1$, then $Q_n-S$ contains exactly one large component of order at least $2^n-|S|-(n+1)$ and the small components having at most $n+1$ vertices in total,

\item \label{cond 3}
If $|S|<g(k)$ and $n+2\leqslant k\leqslant 2n-4$, then $Q_n-S$ contains exactly one large component of order at least $2^n-|S|-(k-1)$ and the small components having at most $k-1$ vertices in total.
\end{enumerate}
\end{Lem}

If $Q_n$ is a spanning subgraph of a graph $G$, then the large component of $Q_n-S$ is contained in the large component of $G-S$ for a vertex subset $S$ of $Q_n$. So from Lemma \ref{Qn}, we have the following result.

\begin{Rem}\label{FA}
\rm{Lemma \ref{Qn} also holds for graphs $FQ_n$ and $AQ_n$.}
\end{Rem}

\section{The Star-structure Connectivity of Folded hypercube}
In this section, to prove $\kappa (FQ_n; K_{1,m})$ and $\kappa^s (FQ_n; K_{1,m})$, we first supply some lemmas for later use.
\begin{Lem}{\rm\cite {16}}
Any two vertices in $FQ_{n}$ have exactly two common neighbors for $n\geqslant 4$ if they have any.
\end{Lem}

\begin{Lem}{\rm\cite {11}}For any positive integer $n$, we have the following statements:
\begin{enumerate}[(i)]
            \item \label{cond 1}
             $\kappa(FQ_{n})=n+1$,

            \item \label{cond 2}
              $FQ_{n}$ is a bipartite graph if and only if $n$ is odd, and

            \item \label{cond 3}
              If $FQ_{n}$ contains an odd cycle, then a shortest odd cycle has the length  $n+1$.
         \end{enumerate}
\end{Lem}

\begin{Lem}{\rm\cite {19}}\label{F 1-N}
Let $K_{1,m}$ be a star in $FQ_{n}$ with $n\geqslant 4$ and $2\leqslant m\leqslant n+1$. If $u$ is a vertex in $FQ_{n}-K_{1,m}$, then
$|N_{FQ_{n}}(u)\cap V(K_{1,m})|\leqslant 2$, and equality holds if and only if $u$ is adjacent to exactly two leaves of $K_{1,m}$.
\end{Lem}

\begin{Lem}{\rm\cite {19}}\label{F 2-N}
Let $K_{1,m}$ be a star in $FQ_{n}$ with $n\geqslant 5$ and $2\leqslant m\leqslant n+1$. If $C$ is a connected subgraph in $FQ_{n}-K_{1,m}$ with $|V(C)|=k\geqslant 2$, then
$|N_{FQ_{n}}(C)\cap V(K_{1,m})|\leqslant 2(k-1)$, and equality holds only  if $C$ is a star in $FQ_{n}$.
\end{Lem}





From Remark \ref{FA} and Lemmas \ref{Qn}, \ref{F 1-N} and \ref{F 2-N}, we obtain a lower bound of $\kappa^{s}(FQ_{n};K_{1,m})$ as follows.
\begin{Lem}\label{FL}
If $n\geqslant 5$, and $2\leqslant m\leqslant n-2$ or $m=n-1\geqslant6$, then $\kappa^{s}(FQ_{n};K_{1,m})\geqslant \lceil\frac{n+1}{2}\rceil$.
\end{Lem}
\begin{proof}
Suppose to the contrary that $\kappa^{s}(FQ_{n};K_{1,m})<\lceil\frac{n+1}{2}\rceil$. Then there exists a $K_{1,m}$-substructure cut $\mathcal{F}$ of $FQ_{n}$ such that $|\mathcal{F}|\leqslant\lceil\frac{n+1}{2}\rceil-1$. It is easy to find $|V(\mathcal{F})|\leqslant (m+1)(\lceil\frac{n+1}{2}\rceil-1)\leqslant\frac{n(m+1)}{2}$.
 We distinguish two cases as follows.

\textbf{Case $1$}. $FQ_{n}-V(\mathcal{F})$ contains an isolated vertex, say $u$.

In this case, $V(\mathcal{F})$ contains the $n+1$ neighbors of $u$. By Lemma \ref{F 1-N}, we have $|N_{FQ_{n}}(u)\cap V(K_{1,m'})|\leqslant 2$ for each member $K_{1,m'}$$(0\leqslant m'\leqslant m)$ in $\mathcal{F}$. So we obtain that
\begin{eqnarray}\nonumber
n+1=|N_{FQ_{n}}(u)|\leqslant \sum_{K_{1,m'}\in \mathcal{F}}|N_{FQ_{n}}(u)\cap V(K_{1,m'})|\leqslant 2|\mathcal{F}|\leqslant 2(\lceil\frac{n+1}{2}\rceil-1)\leqslant n,
\nonumber
\end{eqnarray} a contradiction.

\textbf{Case $2$}. $FQ_{n}-V(\mathcal{F})$ contains no isolated vertex.

Let $C$ be a smallest component of $FQ_{n}-V(\mathcal{F})$ with $k:=|V(C)|$. Then $k\geqslant 2$. For an edge $uv$ in $C$, since $FQ_n$ is triangle-free for $n\geqslant 5$, $|N_{FQ_n}(\{u,v\})|=2n.$ Then we have $|N_{C}(\{u,v\})|\leqslant k-2$ and $|N_{FQ_n}(\{u,v\})\cap V(\mathcal{F})|\geqslant 2n-(k-2)$. Since $|N_{FQ_{n}}(\{u,v\})\cap V(K_{1,m'})|\leqslant 2$ for $0\leqslant m'\leqslant m$ by Lemma \ref{F 2-N}, $\lceil\frac{n+1}{2}\rceil>|\mathcal{F}|\geqslant \lceil\frac{2n-(k-2)}{2}\rceil$. Then $k> n+1$.

For $m\leqslant n-3$ and $n\geqslant 5$,
$$
\aligned
f(m+1)-1-|V(\mathcal{F})|\geqslant&-\frac{(m+1)^2}{2}+(n-\frac{1}{2})(m+1)+1-1-\frac{n(m+1)}{2}\\
=&\frac{1}{2}[-(m+1)^2+(m+1)(2n-1)-n(m+1)]\\
=&\frac{1}{2}[-(m+1)^2+(m+1)(n-1)]\\
=&\frac{1}{2}(m+1)(n-2-m)>0,
\endaligned
$$
which implies that $|V(\mathcal{F})|<f(m+1)-1$. Then by Remark \ref{FA} and Lemma \ref{Qn} $(i)$, $FQ_n-V(\mathcal{F})$ contains exactly one large component of order at least $2^n-|V(\mathcal{F})|-m$ with the small components having at most $m$ vertices in total. Thus $k\leqslant m\leqslant n-3$, a contradiction.

For $m=n-2$, we have $|V(\mathcal{F})|\leqslant\frac{n(m+1)}{2}=\frac{n(n-1)}{2}$ and
$$
\aligned
f(n)-|V(\mathcal{F})|\geqslant& -\frac{n^2}{2}+(n-\frac{1}{2})n+1-\frac{n(n-1)}{2}\\
=&-\frac{n^2}{2}+n^2-\frac{n}{2}+1-\frac{n^2}{2}+\frac{n}{2}=1>0,
\endaligned
$$
that is, $|V(\mathcal{F})|<f(n)$. Then by Remark \ref{FA} and Lemma \ref{Qn} $(ii)$, $FQ_n-V(\mathcal{F})$
contains exactly one large component of order at least $2^n-|V(\mathcal{F})|-(n+1)$ with the small components having at most $n+1$ vertices in total. Thus $k\leqslant n+1$, a contradiction.

For $m=n-1$ and $n>6$, $|V(\mathcal{F})|\leqslant\frac{n(m+1)}{2}=\frac{n^2}{2}$ and
$$
\aligned
g(n+2)-|V(\mathcal{F})|\geqslant&-\frac{(n+2)^2}{2}+(2n-\frac{3}{2})(n+2)+2-n^2-\frac{n^2}{2}\\
=&\frac{1}{2}[-(n+2)^2+(4n-3)(n+2)+4-2n^2-n^2]\\
=&\frac{1}{2}[(n+2)(3n-5)+4-3n^2]=\frac{1}{2}(n-6)>0,
\endaligned
$$
which implies that $|V(\mathcal{F})|<g(n+2)$. Then by Remark \ref{FA} and Lemma \ref{Qn} $(iii)$, $FQ_n-V(\mathcal{F})$ contains exactly one large component of order at least $2^n-|V(\mathcal{F})|-(n+1)$ with the small components having at most $n+1$ vertices in total. Thus $k\leqslant n+1$, a contradiction.
\end{proof}

\begin{Lem}\label{FU}
For $2\leqslant m \leqslant n+1$, $n\geqslant3$, $\kappa(FQ_{n};K_{1,m})\leqslant \lceil\frac{n+1}{2}\rceil$.
\end{Lem}
\begin{proof}
Let $u=00 \ldots 0$ be a vertex in $FQ_{n}$.  Then $N_{FQ_{n}}(u)=\{\overline{u}\}\cup\{u^{i}|1\leqslant i\leqslant n\}$. We assume that the elements in $\{1, 2,\ldots, n\}$ are taken arithmetic operations on module $n$. For any integer $i$ with $1\leqslant i\leqslant \lfloor\frac{n}{2}\rfloor$, we set
\begin{equation*}\label{}S_{i}=
\left\{ \begin{aligned}
&\{u^{2i-1},u^{2i},u^{2i-1,2i}\}\cup \{u^{2i-1,2i,2i+j}|1\leqslant j\leqslant m-2\}, &\mbox{if}~~ m\leqslant n,\\
&\{u^{2i-1},u^{2i},u^{2i-1,2i}\}\cup \{u^{2i-1,2i,2i+j}|1\leqslant j\leqslant n-2\}\cup \{(\overline{u})^{2i-1,2i}\}, &\mbox{if}~~m=n+1.
\end{aligned}\right.
\end{equation*}
Let\begin{equation*}\label{}S_{\lceil\frac{n+1}{2}\rceil}=
\left\{ \begin{aligned}
&\{u^{n},\overline{u},(\overline{u})^{n}\}\cup \{(\overline{u})^{n,j}|1\leqslant j\leqslant m-2\}, &\mbox{if}~~n~ is~odd;\\
&\{\overline{u}, u^{1}, (\overline{u})^{1}\}\cup\{(\overline{u})^{1,j}|2\leqslant j\leqslant m-1\}, &\mbox{if}~~n~is ~even.
\end{aligned}\right.
\end{equation*}

We can find $S_{i}$ induces $K_{1,m}$ with the center $u^{2i-1,2i}$ for $1\leqslant i\leqslant \lfloor\frac{n}{2}\rfloor$, $S_{\frac{n+1}{2}}$ and $S_{\frac{n+2}{2}}$ induce $K_{1,m}$ with the center $\overline{u}^{n}$ and $\overline{u}^{1}$, respectively.

Let $S=\cup_{j=1}^{\lceil\frac{n+1}{2}\rceil}S_{j}$. Then $N_{FQ_{n}}(u)\subseteq V(S)$, which implies that $u$ is an isolated vertex of $FQ_{n}-S$. Thus the set $S$ forms a $K_{1,m}$-structure cut of $FQ_{n}$ with $|S|=\lceil\frac{n+1}{2}\rceil$.
\end{proof}
\begin{figure}[!htbp]
\begin{center}
\includegraphics[height=9cm]{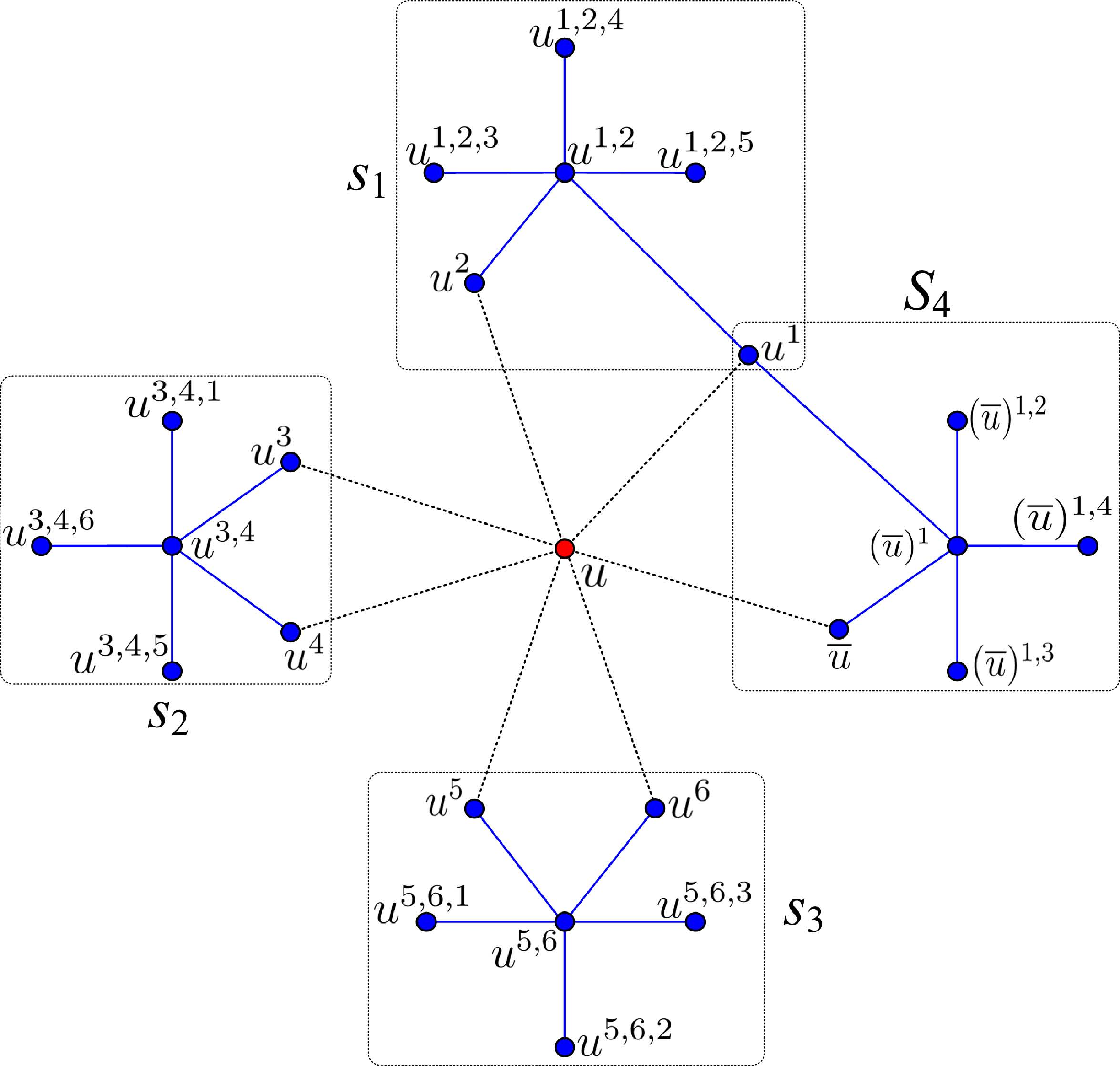}
\caption{\label{}\small{A $K_{1,5}$-structure cut of $FQ_6$}}
\end{center}
\end{figure}
The following theorem is a straightforward result from Lemmas \ref{FL} and \ref{FU}.

\begin{The}\label{FQ}
For $n\geqslant 7$ and $2\leqslant m\leqslant n-1$, we have $\kappa(FQ_n; K_{1,m})
=\kappa^{s}(FQ_{n};K_{1,m})= \lceil\frac{n+1}{2}\rceil$.
\end{The}
In fact, we find the results of Theorem \ref{FQ} also hold for $5\leqslant n\leqslant 6$ and $2\leqslant m\leqslant n-2$ by Lemmas \ref{FL} and \ref{FU}.

\section{The Star-structure Connectivity of Augmented cube}
In this section we determine $\kappa(AQ_n; K_{1,m} )$ and $\kappa^{s}(AQ_{n};K_{1,m})$. We also present some lemmas and properties as follows.

\begin{Lem}\emph{\cite{20}}\label{CN}
Any two vertices in $AQ_n$ have at most 4 common neighbors for $n\geqslant 3$.
\end{Lem}

\begin{Lem}\emph{\cite{20}}\label{LRN}
If $u,w \in V(AQ_{n-1}^i)$, $i\in \{0,1\}$, have a common neighbor in $AQ_{n-1}^{1-i}$, then $w=\overline{u}^{n-1}$ and they have exactly two common neighbors $u^n$ and $\overline{u}^n$ in $AQ_{n-1}^{1-i}$
\end{Lem}

\begin{Lem}\emph{\cite{20}}\label{UUI}
For a hypercube edge $uu^i$, $1\leqslant i\leqslant n$ and $n\geqslant 2$,
\begin{equation*}
\label{}
\ N_{AQ_n}(u)\cap N_{AQ_n}(u^i)=
\left\{ \begin{aligned}
&\{\overline{u}^i, \overline{u}^{i-1}\}, &\mbox{if}~~2\leqslant i\leqslant n,\\
&\{\overline{u}^2, u^{2}\}, &\mbox{if}~~i=1.
\end{aligned}\right.
\end{equation*}
\end{Lem}

\begin{Lem}\emph{\cite{20}}\label{UUCI}
 For a complement edge $u\overline{u}^i$, $2\leqslant i\leqslant n$ and $n\geqslant 2$,
\begin{equation*}
\label{}
\ N_{AQ_n}(u)\cap N_{AQ_n}(\overline{u}^i)=
\left\{ \begin{aligned}
&\{u^i, u^{i+1}, \overline{u}^{i+1},\overline{u}^{i-1}\}, &\mbox{if}~~ 2\leqslant i\leqslant n-1,\\
&\{\overline{u}^{n-1}, u^{n}\}, &\mbox{if}~~i=n.
\end{aligned}\right.
\end{equation*}
\end{Lem}

The following results(Lemmas 4.5-4.7) can give sufficient and necessary condition that any two neighbors of given vertex in $AQ_n$ have exactly four common neighbors, which play an important role in the following discussions. These results are taken from the proof of Lemma 2.8 in \cite {22}, but there is no detailed proof. For convenience, we give direct proofs here.
We set $A_1=\{x^k|1\leqslant k\leqslant n\}$, $A_2=\{\overline{x}^k|2\leqslant k\leqslant n\}$, $B_1=\{y^{t}|1\leqslant t\leqslant n\}$ and $B_2=\{\overline{y}^t|2\leqslant t\leqslant n\}$ for $\{x,y\}\subset V(AQ_n)$. Then $N_{AQ_n}(x)=A_1\cup A_2$ and $N_{AQ_n}(y)=B_1\cup B_2$.
\begin{Lem}\label{UIUJ}
For two hypercube edges $uu^i$ and $uu^j$, $1\leqslant i<j\leqslant n$ and $n\geqslant 3$, $|N_{AQ_n}(u^i)\cap N_{AQ_n}(u^j)|=4$ if and only if
$i\geqslant 2$ and $j=i+1$, or $i=1$ and $j=2,3$. Further, we have
\begin{equation*}
\label{}
\ N_{AQ_n}(u^i)\cap N_{AQ_n}(u^j)=
\left\{ \begin{aligned}
&\{u, u^{1,2}, u^{1,3}, u^{2,3}\}, &\mbox{if}~~i=1~and~j=2, 3;\\
&\{u, u^{i,j}, \overline{u}^{i},\overline{(u^i)}^{j}\}, &\mbox{if}~~i\geqslant 2 ~and~j=i+1;\\
&\{u, u^{i,j}\}, &\mbox{otherwise}.
\end{aligned}\right.
\end{equation*}
\end{Lem}
\begin{proof}
Let $x=u^i$ and $y=u^j$, $1\leqslant i<j\leqslant n.$
It is easy to find that $ A_1\cap B_1=\{u,u^{i,j}\}$.
If $i=1$ and $j=2,3$, then $u^{1,3}\in A_1\cap B_2$ for $j=2$, $u^{1,2}\in A_1\cap B_2$ for $j=3$ and $u^{2,3}\in A_2\cap B_1$ for $j=2,3$. Thus $N_{AQ_n}(u^1)\cap N_{AQ_n}(u^j)=\{u,u^{1,2}, u^{1,3}, u^{2,3}\}$ for $j=2,3$ by Lemma \ref {CN}.
If $i\geqslant 2$ and $j=i+1$,
then $\overline{u}^i \in A_1\cap B_2$, $\overline{(u^i)}^j \in A_2\cap B_1$ for $i=2$ and $\{\overline{u}^i, \overline{(u^i)}^j\} \subseteq A_2\cap B_2$ for $i>2$. Thus $N_{AQ_n}(u^i)\cap N_{AQ_n}(u^j)=\{u,u^{i,j}, \overline{u}^i, \overline{(u^i)}^j\}$ for $i\geqslant 2$ and $j=i+1$ by Lemma \ref {CN}.

Conversely, suppose $i=1$ and $j\geqslant 4$, or $i\geqslant 2$ and $j\geqslant i+2$. For convenience, we set $p\in A_1$, $p'\in A_2$, $q\in B_1$ and $q'\in B_2$, $p_i$ denotes the $i$th position of $p$. We have $A_1\cap B_2=\emptyset$ since $q'$ has at least 3 positions different to $u$ and $p$ has at most 2 positions different to $u$. If $p'$ has two positions different to $u$, then $p'\in\{u^{2,3}, u^{1,3}, u^{1,2}\}$. We find $A_2\cap B_1=\emptyset$ since $\{u^{2,3}, u^{1,3}, u^{1,2}\}\cap B_1=\emptyset$. 
Now, we consider $A_2\cap B_2$.
For $i=1$ and $j\geqslant 4$, we have $A_2\cap B_2=\emptyset$ since $p'_1\neq q'_1$. For $i\geqslant 2$ and $j\geqslant i+2$, suppose $A_2\cap B_2\neq \emptyset$, that is $\overline{(u^i)}^{k'}=\overline{(u^j)}^{t'}$ for some $k', t'$, then $\overline{u}^{t'}=(\overline{u}^{k'})^{i,j}$. It is impossible for $j\geqslant i+2$. Therefore, we have $N_{AQ_n}(u^i)\cap N_{AQ_n}(u^j)=A_1\cap B_1=\{u,u^{i,j}\}$, a contradiction.
\end{proof}

\begin{Lem}\label{UCIUCJ}
For two complement edges $u\overline{u}^i$ and $u\overline{u}^j$, $2\leqslant i<j\leqslant n$ and $n\geqslant 3$, $|N_{AQ_n}(\overline{u}^i)\cap N_{AQ_n}(\overline{u}^j)|=4$ if and only if $j=i+2$. Further, we have
 \begin{equation*} \label{}
 \ N_{AQ_n}(\overline{u}^i)\cap N_{AQ_n}(\overline{u}^j)=
 \left\{ \begin{aligned}
 &\{u, \overline{(\overline{u}^i)}^{j}, \overline{u}^{i+1},(\overline{u}^i)^{j}\}, &\mbox{if}~~j=i+2,\\
 &\{u, \overline{(\overline{u}^i)}^{j}\}, &\mbox{if}~~j=i+1~or~j>i+2.\\ \end{aligned}\right.
 \end{equation*}
 \end{Lem}
\begin{proof} Let $x=\overline{u}^i$ and $y=\overline{u}^j$, $2\leqslant i<j\leqslant n$.
We have $A_2\cap B_2=\{u,\overline{(\overline{u}^i)}^j\}$. If $j=i+2$, then $\{\overline{u}^{i+1}, (\overline{u}^i)^{j}\}\subseteq A_1\cap B_1$ for $k=i+1$, $t=i+2$ and $k=i+2$, $t=i+1$.
Thus $N_{AQ_n}(\overline{u}^i)\cap N_{AQ_n}(\overline{u}^{j})=\{u,\overline{(\overline{u}^i)}^j, \overline{u}^{i+1}, (\overline{u}^i)^{j}\}$ by Lemma \ref{CN}.

Conversely, suppose $j>i$ and $j\neq i+2$. For convenience, we set $p\in A_1$, $p'\in A_2$, $q\in B_1$ and $q'\in B_2$, $p_i$ denotes the $i$th position of $p$. Then $A_1\cap B_2=\emptyset$ since $p_1\neq q'_1$ or $p_2\neq q'_2$. Similarly, $A_2\cap B_1=\emptyset$ since $p'_1\neq q_1$ or $p'_2\neq q_2$. Now, we consider $A_1\cap B_1$. Suppose $A_1\cap B_1\neq \emptyset$, that is $(\overline{u}^i)^k=(\overline{u}^j)^t$ for some $k, t$, then $\overline{u}^{j}=(\overline{u}^{i})^{k,t}$. It is impossible for $j>i$ and $j\neq i+2$.
Therefore, we have $N_{AQ_n}(\overline{u}^i)\cap N_{AQ_n}(\overline{u}^{j})=A_2\cap B_2=\{u, \overline{(\overline{u}^i)}^j\}$, a contradiction.
\end{proof}

\begin{Lem}\label{UIUCJ}
For a hypercube edge and a complement edge $uu^i$ and $u\overline{u}^j$, respectively, $1\leqslant i\leqslant n$, $2\leqslant j\leqslant n$ and $n\geqslant 3$, $|N_{AQ_n}(u^i)\cap N_{AQ_n}(\overline{u}^j)|=4$ if and only if one of the following statements hold: $(i)$ $i=1$ and $j=3$; $(ii)$ $i=j-1$, $3\leqslant j\leqslant n$, or $i=j$, $3\leqslant j\leqslant n$; $(iii)$ $i=j+1$, $2\leqslant j\leqslant n-2$, or $i=j+2$, $2\leqslant j\leqslant n-2$.
    Further, we have
$$ \ N_{AQ_n}(u^i)\cap N_{AQ_n}(\overline{u}^j)=
\left
\{ \begin{array}{ll}
 \{u, u^{1,3}, u^{1,2},u^{2,3}\}, &\mbox{if}~i=1,j=3,\\
 \{u, u^{i,i-1}, \overline{u}^{i-1},(\overline{u}^i)^{i-1}\}, &\mbox{if}~i=j,3\leqslant j\leqslant n,\\
 \{u, u^{i,i+1}, \overline{u}^{i},\overline{(u^i)}^{i+1}\}, &\mbox{if}~i=j+1,2\leqslant j\leqslant n-2,\\ &\mbox{or}~i=j-1,3\leqslant j\leqslant n,\\
 \{u, u^{i,i-1}, \overline{u}^{i-1},\overline{(u^i)}^{i-2}\}, &\mbox{if}~i=j+2, 2\leqslant j\leqslant n-2.\\
  \{u, \overline{(u^i)}^{j}\}, &otherwise.\\
 \end{array}\right.
 $$
 \end{Lem}
 \begin{proof}Let $x=u^{i}$, $1\leqslant i\leqslant n$ and $y=\overline{u}^j$, $2\leqslant j\leqslant n.$
$(i)$ If $i=1$ and $j=3$, then $u\in A_1\cap B_2$, $\{u^{1,2}, u^{1,3}\}\subseteq A_1\cap B_1$ and $u^{2,3}\in A_2\cap B_2$. Thus $N_{AQ_n}(u^1)\cap N_{AQ_n}(\overline{u}^3)=\{u,u^{1,2}, u^{1,3}, u^{2,3}\}$ by Lemma \ref{CN}. $(ii)$ If $i=j-1$, $3\leqslant j\leqslant n$, then $u\in A_1\cap B_2$, $\overline{(u^i)}^{i+1}\in A_2\cap B_1$. For $j=3$, $\{u^{i, i+1}, \overline{u}^i\}\subseteq A_1\cap B_1$. For $j>3$, $u^{i, i+1}\in A_1\cap B_2$, $\overline{u}^i\in A_2\cap B_1$.
Then $N_{AQ_n}(u^i)\cap N_{AQ_n}(\overline{u}^{i+1})=\{u, u^{i,i+1},\overline{u}^{i}, \overline{(u^i)}^{i+1}\}$ by Lemma \ref{CN}. If $i=j$, $3\leqslant j\leqslant n$, then $u\in A_1\cap B_2$, $\overline{u}^{i-1}\in A_2\cap B_1$. For $j=3$, $\{u^{i, i-1}, (\overline{u}^i)^{i-1}\} \subseteq A_1\cap B_1$. For $j>3$, $u^{i, i-1}\in A_1\cap B_2$ and $(\overline{u}^i)^{i-1}\in A_2\cap B_1$. Then $N_{AQ_n}(u^i)\cap N_{AQ_n}(\overline{u}^{i})=\{u, u^{i,i-1},\overline{u}^{i-1}, (\overline{u}^i)^{i-1}\}$ by Lemma \ref{CN}. $(iii)$ If $i=j+1$, $2\leqslant j\leqslant n-2$, then $\{u, u^{i, i+1}\}\subseteq A_1\cap B_2$, $\{\overline{u}^i, \overline{(u^i)}^{i+1}\}\subseteq A_2\cap B_1$. Thus $N_{AQ_n}(u^i)\cap N_{AQ_n}(\overline{u}^{i-1})=\{u, u^{i,i+1},\overline{u}^{i}, \overline{(u^i)}^{i+1}\}$ by Lemma \ref{CN}. If $i=j+2$, $2\leqslant j\leqslant n-2$, then $\{u, u^{i, i-1}\}\subseteq A_1\cap B_2$, $\{\overline{u}^{i-1}, \overline{(u^i)}^{i-2}\}\subseteq A_2\cap B_1$. Thus $N_{AQ_n}(u^i)\cap N_{AQ_n}(\overline{u}^{i-2})=\{u, u^{i,i-1},\overline{u}^{i-1}, \overline{(u^i)}^{i-2}\}$ by Lemma \ref{CN}.

Conversely, suppose $i=1$, $j=2$ or $i=j=2$ or $i=n$, $j=n-1$ or $i\leqslant j-2$ or $i\geqslant j+3$. For $i=1$, $j=2$ or $i=j=2$ or $i=n$, $j=n-1$, it is easy to verify that $N_{AQ_n}(u^i)\cap N_{AQ_n}(\overline{u}^j)=\{u,\overline{(u^i)}^j\}$.
 Then we consider $i\leqslant j-2$, $4\leqslant j\leqslant n$, or $i\geqslant j+3$, $2\leqslant j\leqslant n-3$.  For convenience, We set $p\in A_1$, $p'\in A_2$, $q\in B_1$ and $q'\in B_2$. If $q$ has two positions different to $u$, then $j=3$ and $q\in\{u^{1,3}, u^{1,2}, u^{2,3}\}$. Thus $A_1\cap B_1=\emptyset$ for $i\geqslant j+3$ or $ i\leqslant j-2$.
We have $A_2\cap B_2=\emptyset$ since $p'_2\neq q'_2$ or $p'_1\neq q'_1$. Now, we consider $A_1\cap B_2$ and $A_2\cap B_1$. Suppose $A_1\cap B_2\neq \emptyset$, that is $u^{i,k}=\overline{(\overline{u}^j)}^{t'}$ for some $k, t'$, then $k=i$ and $t'=j$, thus $A_1\cap B_2 =\{u\}$. Suppose $A_2\cap B_1\neq \emptyset$, that is $\overline{(u^i)}^{k'}=(\overline{u}^j)^t$ for some $k', t$, then $k'=j$ and $t=i$, thus $A_2\cap B_1 =\{\overline{(u^i)}^j\}$. Therefore, we have $N_{AQ_n}(u^i)\cap N_{AQ_n}(\overline{u}^j)=\{u,\overline{(u^i)}^j\}$.
Hence $N_{AQ_n}(u^i)\cap N_{AQ_n}(\overline{u}^{j})=\{u, \overline{(u^i)}^{j}\}$ for $i=1$, $j=2$ or $i=j=2$ or $i=n$, $j=n-1$ or $i\leqslant j-2$ or $i\geqslant j+3$, a contradiction.
\end{proof}
\begin{figure}[!htbp]
\centering
\subfigure[Graph $T$]{\includegraphics[totalheight=5.5cm]{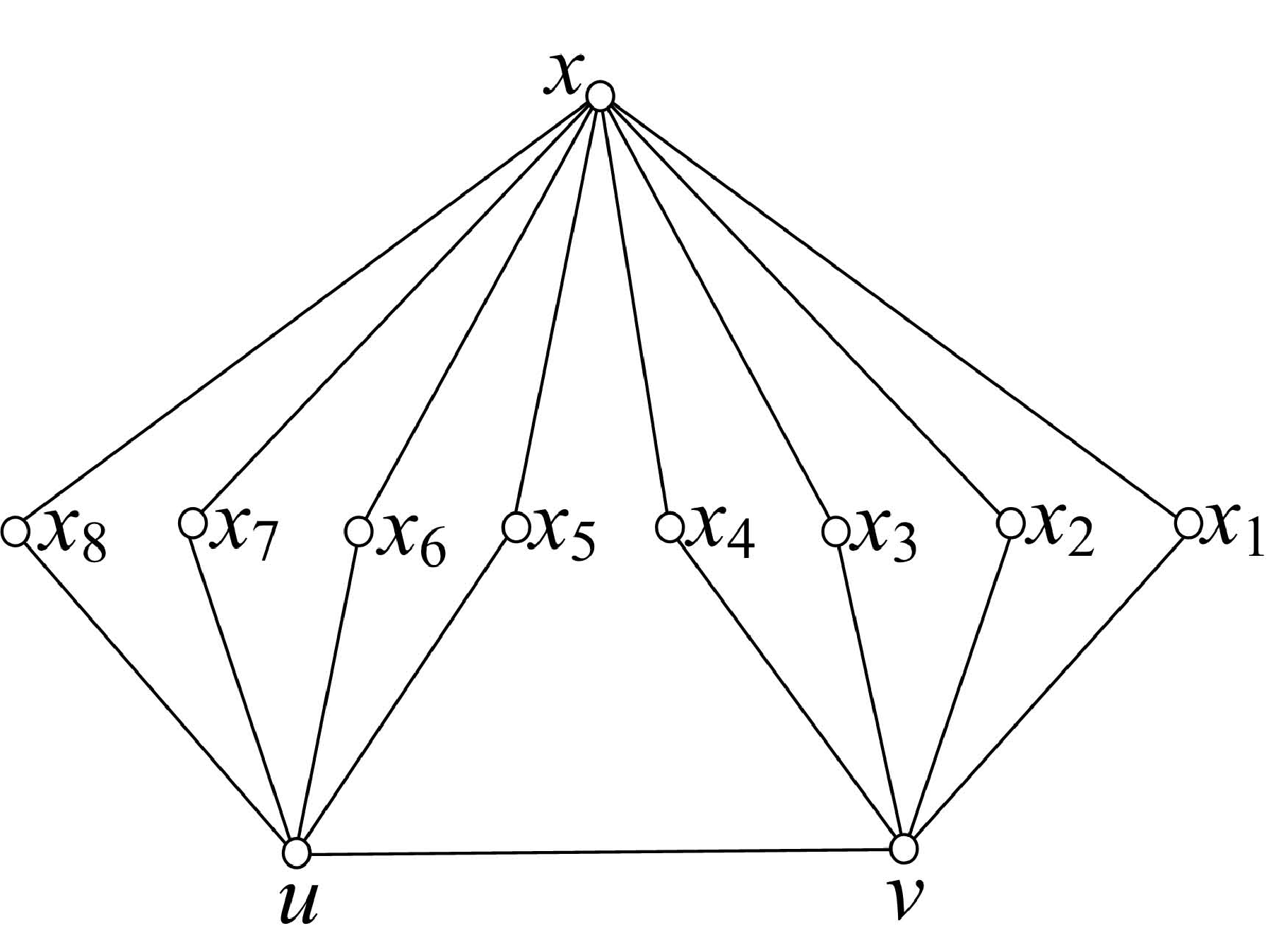}}~~~
\subfigure[Graph $H$]{\includegraphics[totalheight=5.5cm]{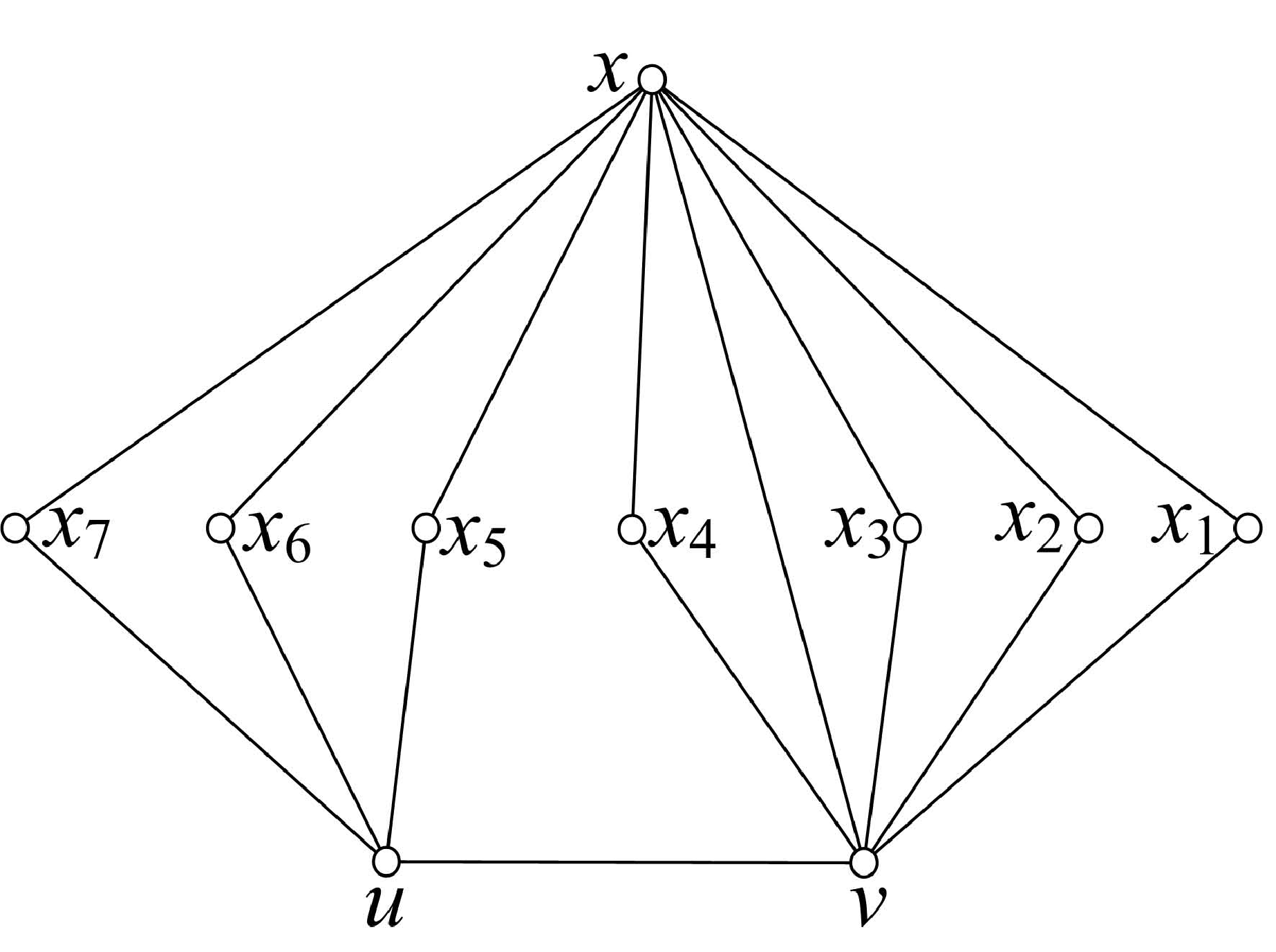}}
\caption{Illustration for Lemmas 4.8 and 4.9.}
\end{figure}

In following, $AQ_n$ has no subgraph isomorphic to $T$(resp.$H$) implies that $T$(resp.$H$) is not a subgraph of $AQ_n$.
\begin{Lem}\label{T}
Let $T$ be a graph with $V(T)=\{x,u,v\}\cup\{x_i|1\leqslant i\leqslant 7\}$ and $E(T)=\{uv,vx_1,vx_2, vx_3, vx_4, ux_5, ux_6, ux_7, ux_8\}\cup\{xx_i|1\leqslant i\leqslant 8\}$\emph{(see Figure 3 (a))}. Then $AQ_n$ has no subgraph isomorphic to $T$.
\end{Lem}
\begin{proof}We prove it by induction on $n$. For $n\leqslant4$, it is trivial since $AQ_n$ is $(2n-1)$-regular and $T$ has a 8-degree vertex $x$.
Suppose that $T$ is not a subgraph of $AQ_{n-1}$ for $n\geqslant 5$, but $T$ is a subgraph of $AQ_{n}$ for $n\geqslant 5$. 
 Then $V(T)\cap V(AQ_{n-1}^0)\neq\emptyset$ and  $V(T)\cap V(AQ_{n-1}^1)\neq\emptyset$. Otherwise, $T$ is a subgraph of some $AQ_n^i$ for $i=0$ or $1$, 
contradicting the induction hypothesis. Without loss of generality, assume $x\in V(AQ_{n-1}^0)$. Then $x$ has at most two neighbors in $AQ_{n-1}^1\cap T$. Since $|N_T(u)|=|N_T(v)|=5$, $|N_{AQ_{n-1}^0}(u)|\geqslant 3$ or $|N_{AQ_{n-1}^0}(v)|\geqslant 3$. Then $u\in V(AQ_{n-1}^0)$ or $v\in V(AQ_{n-1}^0)$, say $u$, which implies that $|N_{AQ_{n-1}^0}(v)|\geqslant 3$. Thus $\{u,v\}\subset V(AQ_{n-1}^0)$.
So there is $x_i\in V(AQ_{n-1}^1)$ for $1\leqslant i\leqslant 8$, which implies that the $x_i$ is a common neighbor of $x$ and $v$, or $x$ and $u$ in $AQ_{n-1}^1$. By Lemma \ref{LRN}, $u=\overline{x}^{n-1}$ or $v=\overline{x}^{n-1}$. If $u=\overline{x}^{n-1}$, then  $\{x_1, x_2, x_3, x_4, u\}\subseteq N_{AQ_n}(x)\cap N_{AQ_n}(v)$, contradicting Lemma \ref{CN}. If $v=\overline{x}^{n-1}$, then  $\{x_5, x_6, x_7, x_8, v\}\subseteq N_{AQ_n}(x)\cap N_{AQ_n}(u)$, also a contradiction. Thus $T$ is not a subgraph of $AQ_n$.
\end{proof}

\begin{Lem}\label{H}
Let $H$ be a graph with $V(H)=\{x,u,v\}\cup\{x_i|1\leqslant i\leqslant 7\}$ and $E(H)=\{uv,vx,vx_1,vx_2,vx_3,vx_4, ux_5, ux_6, ux_7\}\cup\{xx_i|1\leqslant i\leqslant 7\}$\rm{(see Figure 3 (b))}. Then $AQ_n$ has no subgraph isomorphic to $H$.
\end{Lem}
\begin{proof}Suppose to the contrary that $H$ is a subgraph of $AQ_n$.  Since $|N_{H}(x)\cap N_{H}(v)|=4$ and $xv\in E(H)$, $x=\overline{v}^j$ for $2\leqslant j\leqslant n-1$ by Lemmas \ref{UUI} and \ref{UUCI}. Then $N_{H}(x)\cap N_{H}(v)=\{x_1, x_2, x_3, x_4\}=\{v^j, v^{j+1},\overline{v}^{j+1}, \overline{v}^{j-1}\}$.
Since $uv\in E(H)$, we have the following cases according to $u=\overline{v}^i$, $2\leqslant i\leqslant n$, or $u=v^i$, $1\leqslant i\leqslant n$.

\textbf{Case 1.} $u=\overline{v}^i$, $2\leqslant i\leqslant n$.

By Lemma \ref{UCIUCJ}, we have $j=i\pm2$ since $|N_{H}(u)\cap N_{H}(x)|=4$. If $j=i+2$, then $N_{H}(u)\cap N_{H}(x)=
\{v,\overline{(\overline{v}^{i})}^j,\overline{v}^{j-1},
(\overline{v}^i)^{j}\}=\{v,x_5, x_6, x_7\}$. We have
$\{v,x_5, x_6, x_7\}\cap \{x_1, x_2, x_3, x_4\}=\{\overline{v}^{j-1}\}$, a contradiction.
If $j=i-2$, then $N_{H}(u)\cap N_{H}(x)= \{v,\overline{(\overline{v}^{j})}^i,\overline{v}^{j+1},
(\overline{v}^j)^{i}\}=\{v,x_5, x_6, x_7\}$.  Thus
$\{v,x_5, x_6, x_7\}\cap \{x_1, x_2, x_3, x_4\}=\{\overline{v}^{j+1}\}$, a contradiction.

\textbf{Case 2.} $u=v^i$, $1\leqslant i\leqslant n$

Since $|N_{H}(u)\cap N_{H}(x)|=4$, we have the following two subcases by Lemma \ref{UIUCJ}.

\textbf{Subcase 2.1.} $i=1$ and $j=3$, or $i=j$ for $3\leqslant j\leqslant n$.

If $i=1$ and $j=3$, then $N_{H}(u)\cap N_{H}(x)=\{v,v^{1,3}, v^{1,2}, v^{2,3}\}$. That is, $\{v,x_5, x_6, x_7\}=\{v,v^{1,3}, v^{1,2}, v^{2,3}\}$. Thus $\{v,x_5, x_6, x_7\}\cap \{x_1, x_2, x_3, x_4\}=\{\overline{v}^{2}\}$, a contradiction.
 If $i=j$, $3\leqslant j\leqslant n$, then $N_{H}(v^i)\cap N_{H}(\overline{v}^j)=\{v,v^{j,j-1}, \overline{v}^{j-1}, (\overline{v}^j)^{j-1}\}=\{v,x_5, x_6, x_7\}$.
Thus
$\{v,x_5, x_6, x_7\}\cap \{x_1, x_2, x_3, x_4\}=\{\overline{v}^{j-1}\}$, a contradiction.

\textbf{Subcase 2.2.} $i=j-1$ for $3\leqslant j\leqslant n$, or $i=j+1$ for $2\leqslant j\leqslant n-2$, or $i=j+2$.

If $i=j-1$ for $3\leqslant j\leqslant n$ or $i=j+1$ for $2\leqslant j\leqslant n-2$, then $N_{H}(u)\cap N_{H}(x)=\{v,v^{i,i+1}, \overline{v}^{i}, \overline{(v^i)}^{i+1}\}$. Thus $\{v,x_5, x_6, x_7\}=\{v,v^{j-1,j}, \overline{v}^{j-1}, \overline{(v^{j-1})}^{j}\}$ for $i=j-1$ and $\{v,x_5, x_6, x_7\}=\{v,v^{j+1,j+2}, \overline{v}^{j+1},\overline{(v^{j+1})}^{j+2}\}$ for $i=j+1$.
It is easy to find $\{v,x_5, x_6, x_7\}\cap \{x_1, x_2, x_3, x_4\}=\{\overline{v}^{j\mp1}\}$ for $i=j\mp1$, a contradiction.
If $i=j+2$, then $N_{H}(u)\cap N_{H}(x)=\{v,v^{i,i-1}, \overline{v}^{i-1}, \overline{(v^i)}^{i-2}\}=\{v,v^{j+2,j+1}, \overline{v}^{j+1}, \overline{(v^{j+2})}^{j}\}=\{v,x_5, x_6, x_7\}$.
 Thus
$\{v,x_5, x_6, x_7\}\cap \{x_1, x_2, x_3, x_4\}=\{\overline{v}^{j+1}\}$, a contradiction.
\end{proof}

\begin{Lem}\label{K1M}
Let $K_{1,m}$ be a star in $AQ_{n}$ with $n\geqslant 5$ and $2\leqslant m\leqslant 2n-1$. If $uv$ is an edge in $FQ_{n}-K_{1,m}$, then
$|N_{AQ_n}(\{u,v\})\cap V(K_{1,m})|\leqslant 7$.
\end{Lem}
\begin{proof}Suppose to the contrary that there exists a $K_{1,m}$ such that $|N_{AQ_n}(\{u,v\})\cap V(K_{1,m})| \geqslant 8$. Let $V(K_{1,m})=\{x\}\cup\{x_i|1\leqslant i\leqslant m\}$ and $E(K_{1,m})=\{xx_i|1\leqslant i\leqslant m\}$.

If $ux$ or $vx\in E(AQ_n)$, without loss of generality, assume $vx\in E(AQ_n)$, then $v$ is a common neighbor of $u$ and $x$. We have $|\{x_i|1\leqslant i\leqslant m\}\cap N_{AQ_n}(u)|\leqslant 3$ and $|\{x_i|1\leqslant i\leqslant m\}\cap N_{AQ_n}(v)|\leqslant 4$ by Lemma \ref{CN}. Since $|N_{AQ_n}(\{u,v\})\cap V(K_{1,m})|\geqslant 8$, $|\{x_i|1\leqslant i\leqslant m\}\cap N_{AQ_n}(u)|=3$ and $|\{x_i|1\leqslant i\leqslant m\}\cap N_{AQ_n}(v)|=4$. Then we set $\{vx_i|1\leqslant i\leqslant 4\}\subset E(AQ_n)$ and $\{ux_i|5\leqslant i\leqslant 7\}\subset E(AQ_n)$. Let $H$ be a graph with $V(H)=\{x,u,v\}\cup\{x_i|1\leqslant i\leqslant 7\}$ and $E(H)=\{uv,vx,vx_1,vx_2,vx_3,vx_4, ux_5, ux_6, ux_7\}\cup\{xx_i|1\leqslant i\leqslant 7\}$. Thus $H$ is not a subgraph of $AQ_n$ by Lemma \ref{H}, a contradiction.

If $ux\notin E(AQ_n)$ and $vx\notin E(AQ_n)$, then we have $|\{x_i|1\leqslant i\leqslant m\}\cap N_{AQ_n}(u)|=4$ and $|\{x_i|1\leqslant i\leqslant m\}\cap N_{AQ_n}(v)|=4$ by Lemma \ref{CN}. Thus we set $\{vx_i|1\leqslant i\leqslant 4\}\subset E(AQ_n)$ and $\{ux_i|5\leqslant i\leqslant 8\}\subset E(AQ_n)$. Let $T$ be a graph with $V(T)=\{x,u,v\}\cup \{x_i|1\leqslant i\leqslant 8\}$ and $E(T)=\{uv\}\cup\{ xx_k|1\leqslant k\leqslant 8\}\cup\{ vx_i|1\leqslant i\leqslant 4,\}\cup\{ux_j|5\leqslant j\leqslant 8\}$. Then $T$ is not a subgraph of $AQ_n$ by Lemma \ref{T}, a contradiction.
\end{proof}
\begin{Rem}
{\rm To show the result of Lemma \ref{K1M} is sharp, we see the following example. See Figure \ref{rem}. for an illustration. For $2\leqslant i\leqslant n-2$, we set $v=\overline{x}^i$, $u=\overline{(\overline{x}^i)}^{i+2}$ and $V(K_{1,6})=\{x,x^i,x^{i+1}, \overline{x}^{i+1}, \overline{x}^{i-1}, x^{i+2}, \overline{x}^{i+2}\}$, $E(K_{1,6})=\{xx^i,xx^{i+1}, x\overline{x}^{i+1}, x\overline{x}^{i-1}, xx^{i+2},\\ x\overline{x}^{i+2}\}$. Then we can find $\{vx, vx^i,vx^{i+1}, v\overline{x}^{i+1}, v\overline{x}^{i-1}, vu\}\subset E(AQ_n)$.  We know $ux^{i+2}$ is a hypercube edge of dimension $i+1$ and $u\overline{x}^{i+2}$ is a complement edge of dimension $i$, respectively.
Then $|N_{AQ_n}(\{u,v\})\cap V(K_{1,6})|=|\{x,x^i,x^{i+1}, \overline{x}^{i+1}, \overline{x}^{i-1}, x^{i+2}, \overline{x}^{i+2}\}|=7$. Thus, the result of Lemma \ref{K1M} is optimal.}
\end{Rem}
\begin{figure}[!htbp]
\begin{center}
\includegraphics[totalheight=6cm]{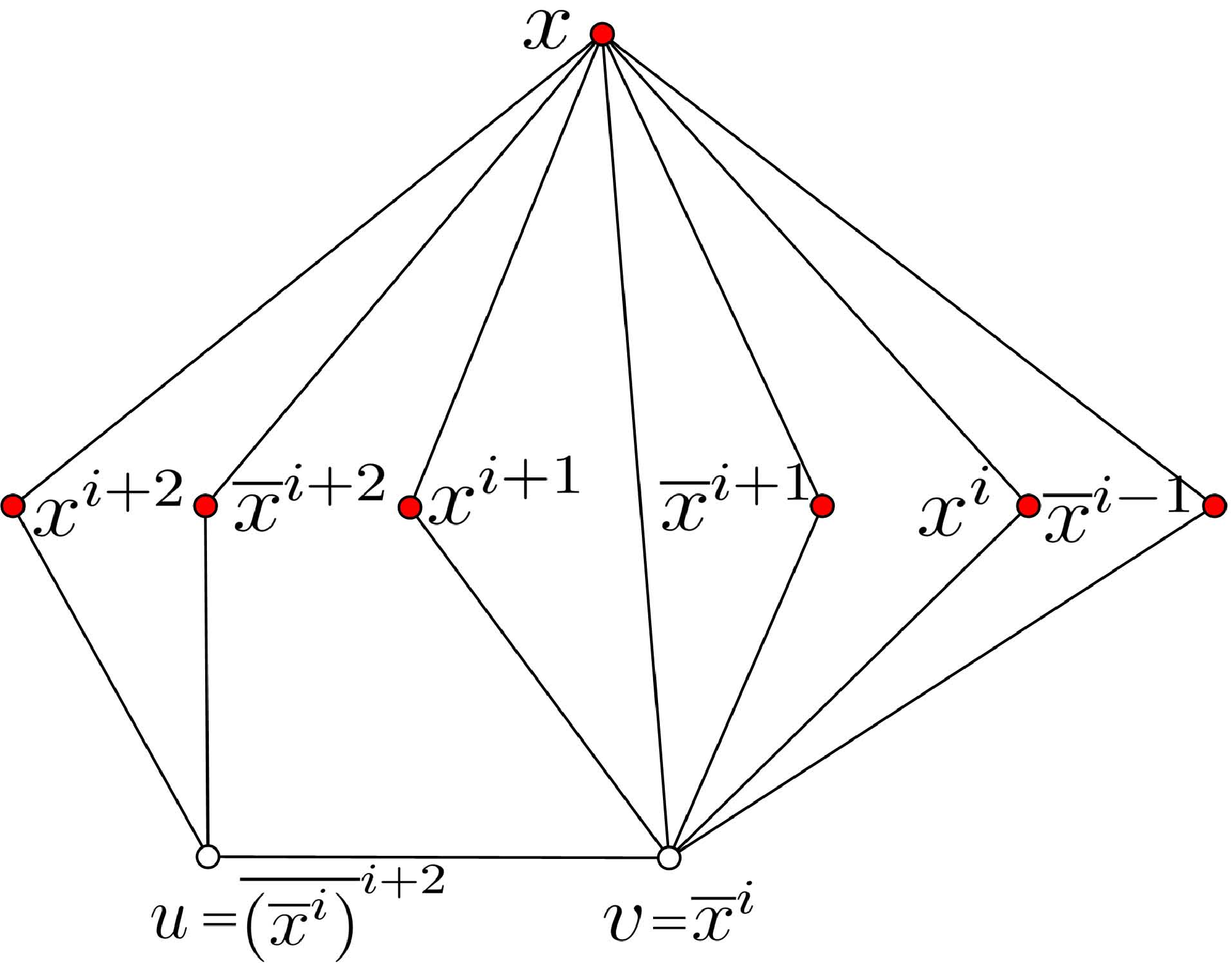}
\caption{\label{rem}\small{Illustration for Remark 4.12}}
\end{center}
\end{figure}

\begin{Lem}\emph{\cite{23}}\label{SK1M}
Let $F$ be a $K_{1,m}$-substructure cut of $AQ_n$ with $n\geqslant 6$. If there exists an isolated vertex in $AQ_n-V(F)$, then $|F|\geqslant \lceil\frac{n-1}{2}\rceil$.
\end{Lem}

\begin{Lem}\label{AL}
For integers $m,n$ with $ 4\leqslant m\leqslant \frac{3n-15}{4}$, $\kappa^s(AQ_n; K_{1,m})\geqslant \lceil\frac{n-1}{2}\rceil$.
\end{Lem}
\begin{proof}Suppose to the contrary that $\kappa^{s}(AQ_{n};K_{1,m})<\lceil\frac{n-1}{2}\rceil$. Then there exists a set $\mathcal{F}$ of subgraphs of $AQ_{n}$ that each is a star of at most $m$ leaves so that
$|\mathcal{F}|\leqslant\lceil\frac{n-1}{2}\rceil-1$ and $AQ_{n}-V(\mathcal{F})$ is disconnected or a vertex. Let $C$ be a smallest component of $AQ_{n}-V(\mathcal{F})$ with $k:=|V(C)|$. If $k=1$, then $|\mathcal{F}|\geqslant \lceil\frac{n-1}{2}\rceil$ by Lemma \ref{SK1M}, a contradiction. Hence $k\geqslant 2$.

For an edge $uv$ in $C$, $|N_{AQ_n}(\{u,v\})|\geqslant4n-8$ by Lemma \ref{CN}.  Then we have $|N_{C}(\{u,v\})|\leqslant k-2$ and $|N_{AQ_n}(\{u,v\})\cap V(\mathcal{F})|\geqslant 4n-8-(k-2)$. Since $|N_{AQ_n}(\{u,v\})\cap V(K_{1,m'})|\leqslant 7$ for $0\leqslant m'\leqslant m$ by Lemma \ref{K1M}, $\lceil\frac{n-1}{2}\rceil>|\mathcal{F}|\geqslant\lceil\frac{4n-8-(k-2)}{7}\rceil$. Then $n<2k+5$. It is easy to find $|V(\mathcal{F})|\leqslant (m+1)(\lceil\frac{n-1}{2}\rceil-1)\leqslant\frac{(n-2)(m+1)}{2}$.

Since $m\leqslant \frac{3n-15}{4}$, $n\geqslant \frac{4m}{3}+5$ and $\frac{2m}{3}+1\leqslant n-2$. Then we have
$$
\aligned
f(\frac{2m}{3}+1)-1-|V(\mathcal{F})|\geqslant&-\frac{1}{2}(\frac{2m}{3}+1)^2+
(n-\frac{1}{2})(\frac{2m}{3}+1)+1-1-\frac{(n-2)(m+1)}{2}\\
=&\frac{1}{18}[-(2m+3)^2+3(2n-1)(2m+3)-9(n-2)(m+1)]\\
=&\frac{1}{18}[(3m+9)n-4m^2)]\\
\geqslant&\frac{1}{18}[(3m+9)(\frac{4m}{3}+5)-4m^2]\\
=&\frac{1}{18}(27m+45)>0,
\endaligned
$$
which implies that $|V(\mathcal{F})|< f(\frac{2m}{3}+1)-1$. Then $AQ_n-V(\mathcal{F})$ contains exactly a large component with the small components having at most $\frac{2m}{3}$ vertices in total by Remark \ref{FA} and Lemma \ref{Qn} $(i)$. Thus $k\leqslant \frac{2m}{3}$ and $2k+5\leqslant \frac{4m}{3}+5\leqslant n$, a contradiction.
\end{proof}

\begin{Lem}\label{AU}
For integers $m,n$ with $4\leqslant m\leqslant 2n-2$, $\kappa(AQ_n; K_{1,m})\leqslant \lceil\frac{n-1}{2}\rceil$.
\end{Lem}
\begin{proof}Let $u=00\ldots 0$ be a vertex of $AQ_n$. Then $N_{AQ_n}(u)=\{u^i|1\leqslant i\leqslant n\}\cup\{ \overline{u}^i|2\leqslant i\leqslant n\}$.  We need to show a $K_{1,m}$-structure cut with size $\lceil\frac{n-1}{2}\rceil$. According to the parity of $m$, we have the following two cases:

\textbf{Case 1.} $m$ is odd.

We set $S_1$ and $S_{\lceil\frac{n-1}{2}\rceil}$ as follows.
\begin{equation*}
S_1=\{u^1, u^2, \overline{u}^2, u^3, \overline{u}^3, (\overline{u}^2)^n\}\cup \{(\overline{u}^2)^i, \overline{(\overline{u}^2)}^i|4\leqslant i\leqslant \frac{m+1}{2}\},
\end{equation*}
and for $n$ is odd,
\begin{equation*}
S_{\lceil\frac{n-1}{2}\rceil}
=\{u^{n-1}, \overline{u}^{n-1}, u^{n}, \overline{u}^{n}, (\overline{u}^{n-1})^1, (\overline{u}^{n-1})^{n-1}\}\cup \{(\overline{u}^{n-1})^i, \overline{(\overline{u}^{n-1})}^{i}|2\leqslant i\leqslant \frac{m-3}{2}\};
\end{equation*}
for $n$ is even,
\begin{equation*}
S_{\lceil\frac{n-1}{2}\rceil}
=\{u^{n}, \overline{u}^{n}, (\overline{u}^{n})^1, (\overline{u}^{n})^{n}, (\overline{u}^{n})^{n-2}, (\overline{u}^{n})^{n-1} \}\cup \{(\overline{u}^{n})^i, \overline{(\overline{u}^{n})}^{i}|2\leqslant i\leqslant \frac{m-3}{2} \}.
\end{equation*}
Now we consider $S_k$ for $2\leqslant k\leqslant \lceil\frac{n-1}{2}\rceil-1$. If $m\leqslant 2n-4k+3$, then
\begin{equation*}
S_k=\{u^{2k}, \overline{u}^{2k}, u^{2k+1}, \overline{u}^{2k+1}, (\overline{u}^{2k})^1, (\overline{u}^{2k})^{2k}\}\cup\{(\overline{u}^{2k})^i, \overline{(\overline{u}^{2k})}^{i}|2k+2\leqslant i\leqslant 2k+2+\frac{m-7}{2}\};
\end{equation*}
and if $m\geqslant 2n-4k+5$, then
$$
\aligned
S_k=&\{u^{2k}, \overline{u}^{2k}, u^{2k+1}, \overline{u}^{2k+1}, (\overline{u}^{2k})^1, (\overline{u}^{2k})^{2k}\}\cup\{(\overline{u}^{2k})^i, \overline{(\overline{u}^{2k})}^{i}|2k+2\leqslant i\leqslant n\}\\&\cup \{(\overline{u}^{2k})^j, \overline{(\overline{u}^{2k})}^{j}|2\leqslant j\leqslant 2k-n+\frac{m-1}{2}\}.
\endaligned
$$

\textbf{Case 2.} $m$ is even.

Let
\begin{equation*}
S_1=\{u^1, u^2, \overline{u}^2, u^3, \overline{u}^3\}\cup \{(\overline{u}^2)^i, \overline{(\overline{u}^2)}^i|4\leqslant i\leqslant \frac{m+2}{2}\}.
\end{equation*}
We set $S_{\lceil\frac{n-1}{2}\rceil}$ as follows. For $n$ is odd, $S_{\lceil\frac{n-1}{2}\rceil}=\{u^{n-1}, \overline{u}^{n-1}, u^{n}, \overline{u}^{n}, (\overline{u}^{n-1})^1\}$ with $m=4$;
\begin{equation*}
S_{\lceil\frac{n-1}{2}\rceil}=\{u^{n-1}, \overline{u}^{n-1}, u^{n}, \overline{u}^{n}, (\overline{u}^{n-1})^1, (\overline{u}^{n-1})^{n-2}, (\overline{u}^{n-1})^{n-1}\}\cup \{\overline{u}^{n-1})^j, \overline{(\overline{u}^{n-1})}^{j}|2\leqslant j\leqslant \frac{m}{2}-2  \}
\end{equation*}
with $m\geqslant 6$. For $n$ is even, we set
\begin{equation*}
S_{\lceil\frac{n-1}{2}\rceil}=\{u^{n}, \overline{u}^{n}, (\overline{u}^{n})^1, (\overline{u}^{n})^{n}, (\overline{u}^{n})^{n-1} \}\cup \{(\overline{u}^{n})^i, \overline{(\overline{u}^{n})}^{i}|2\leqslant i\leqslant \frac{m-2}{2} \}.
\end{equation*}
Next we consider the case $S_k$ for $2\leqslant k\leqslant \lceil\frac{n-1}{2}\rceil-1$. If $m\leqslant 2n-4k+2$, then
\begin{equation*}
S_k=\{u^{2k}, \overline{u}^{2k}, u^{2k+1}, \overline{u}^{2k+1}, (\overline{u}^{2k})^1\}\cup\{(\overline{u}^{2k})^i, \overline{(\overline{u}^{2k})}^{i}|2k+2\leqslant i\leqslant 2k+2+\frac{m-6}{2}\};
\end{equation*}
and if $m\geqslant 2n-4k+4$, then
$$
\aligned
S_k=&\{u^{2k}, \overline{u}^{2k}, u^{2k+1}, \overline{u}^{2k+1}, (\overline{u}^{2k})^1, (\overline{u}^{2k})^{2k-1}, (\overline{u}^{2k})^{2k}\}\cup\{(\overline{u}^{2k})^i, \overline{(\overline{u}^{2k})}^{i}|2k+2\leqslant i\leqslant n\}\\&\cup \{(\overline{u}^{2k})^j, \overline{(\overline{u}^{2k})}^{j}|2\leqslant j\leqslant 2k-n-1+\frac{m}{2}\}.
\endaligned
$$

We can find $S_{k}$ induces a star $K_{1,m}$ with the center $\overline{u}^{2k}$ for $1\leqslant k\leqslant \lceil\frac{n-1}{2}\rceil$.
Let $S=\cup_{k=1}^{\lceil\frac{n-1}{2}\rceil}S_{k}$. Then $N_{AQ_{n}}(u)\subseteq V(S)$. It is easy to find that $u$ is an isolated vertex of $AQ_{n}-V(S)$. So the set $S$ forms a $K_{1,m}$-structure cut of $AQ_{n}$ with $|S|=\lceil\frac{n-1}{2}\rceil$.
\end{proof}
\begin{figure}[!htbp]
\begin{center}
\includegraphics[totalheight=9cm]{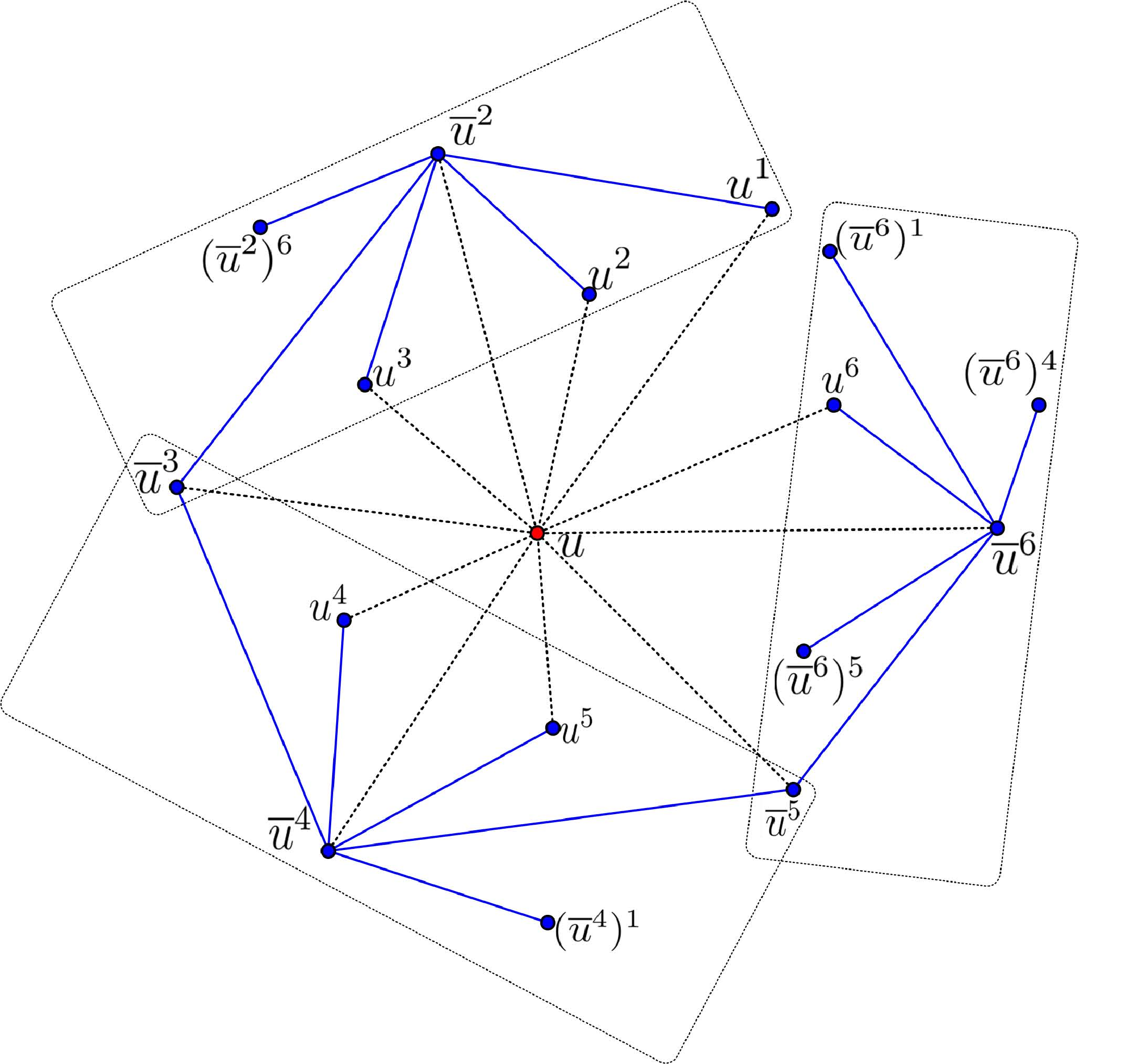}
\caption{\label{}\small{A $K_{1,5}$-structure cuts of $AQ_6$}}
\end{center}
\end{figure}
By the Lemmas \ref{AL}, \ref{AU}, we obtain the following theorem.
\begin{The}For integers $m,n$ with $4 \leqslant m\leqslant \frac{3n-15}{4}$, $\kappa ^s(AQ_n; K_{1,m})=\kappa(AQ_n; K_{1,m})=\lceil\frac{n-1}{2}\rceil$.
\end{The}

\section{Conclusion }
In this paper, we have showed that star-structure connectivity and star-substructure connectivity of $n$-dimensional folded hypercube and  $n$-dimensional augmented cube. The results are summarized as follows.
For $2\leqslant m\leqslant n-1$ and $n\geqslant 7$,
\begin{equation*}
\kappa(FQ_n; K_{1,m} )=\kappa^{s}(FQ_{n};K_{1,m})= \lceil\frac{n+1}{2}\rceil,
\end{equation*}
and for $4\leqslant m\leqslant \frac{3n-15}{4}$,
\begin{equation*}
\kappa ^s(AQ_n; K_{1,m})=\kappa(AQ_n; K_{1,m})=\lceil\frac{n-1}{2}\rceil.
\end{equation*}

It is easy to find that the range of leaves $m$ of star-structure is $1\leqslant m\leqslant n+1$ when considering $\kappa(FQ_n; K_{1,m})$. Sabir et al. \cite{9} established $\kappa^s(FQ_n; K_{1,1})=\kappa(FQ_n; K_{1,1})=n$ for $n\geqslant 7$. We obtain the $\kappa(FQ_n; K_{1,m})$ and $\kappa^s(FQ_n; K_{1,m})$ for $2\leqslant m\leqslant n-1$, $n\geqslant 7$. Then for $n\leqslant m\leqslant n+1$, there is no relevant results. Similarly, the range of leaves $m$ of star-structure is $1\leqslant m\leqslant 2n-1$ when considering $\kappa(AQ_n; K_{1,m})$.
Kan et al. \cite{23} showed  $\kappa^s(AQ_n; K_{1,m})=\kappa(AQ_n; K_{1,m})=\lceil\frac{2n-1}{1+m}\rceil$ for $1\leqslant m\leqslant 3$, and we obtain the $\kappa(AQ_n; K_{1,m})$ and $\kappa^s(AQ_n; K_{1,m})$ for $4\leqslant m\leqslant \frac{3n-15}{4}$. Then there is no relevant results for $\frac{3n-15}{4}<m\leqslant 2n-1$. Hence, we may focus on $\kappa(FQ_n; K_{1,m})$ and $\kappa^s(FQ_n; K_{1,m})$ for $n\leqslant m\leqslant n+1$; $\kappa(AQ_n; K_{1,m})$ and $\kappa^s(AQ_n; K_{1,m})$ for $\frac{3n-15}{4}<m\leqslant 2n-1$.
\section*{Acknowledgement}
This work is supported by NSFC (Grant No. 11871256).

\end{document}